\documentclass{elsarticle}

\usepackage[centertags]{amsmath}
\usepackage{amsfonts}
\usepackage{amssymb}
\usepackage{newlfont}

\newtheorem{theorem}{Theorem}

\newtheorem{problem}{Problem}

\newtheorem{definition}{Definition}

\newenvironment{proof}{\medskip{\em Proof.}}{\par}

\newcommand{\inn}{\mathrm{ int}\,}

\newfont{\gothic}{eufm9 scaled 1200}

\begin{document}

\begin{frontmatter}
\title{Some analogies between Haar meager sets\\ and Haar null sets
       in abelian Polish groups}
\author{Eliza  Jab{\l}o{\'n}ska}
\ead{elizapie@prz.edu.pl}

\address{Department of Mathematics, Rzesz{\'o}w University of
Technology, Powsta\'{n}c\'{o}w~Warszawy~12, 35-959~Rzesz{\'o}w, POLAND}


\begin{abstract}
In the paper we would like to pay attention to some analogies between Haar meager sets and Haar null sets. Among others, we will show that $0\in \inn (A-A)$ for each Borel set $A$, which is not Haar meager in an abelian Polish group. Moreover, we will give an example of a Borel non-Haar meager set $A\subset c_0$ such that $\inn (A+A)=\emptyset$. Finally, we will define  $D$-measurability as a topological analog of Christensen measurability,  and apply our generalization of Piccard's theorem to prove that each  $D$-measurable homomorphism is continuous.  Our results refer to the papers \cite{Ch}, \cite{Darji} and \cite{FS}.\end{abstract}
\end{frontmatter}

\section{Introduction}
In 1972 J.P.R.~Christensen defined \textit{Haar null} sets in an abelian Polish group (a topological abelian group with a complete separable
metric) in such a way that in a locally compact group it is equivalent to the notation of Haar measure zero sets.
These definition has been extended further to nonabelian
groups by J.~Mycielski \cite{Myc}.
Unaware of  Christensen's result, B.R.~Hunt, T.~Sauer and J.A.~Yorke \cite{HSY}-\cite{HSY1} found this notation again, but in a topological abelian group with a complete metric (not necessary separable).

In 2013 U.B.~Darji introduced another family of "small" sets in an abelian Polish group, which is equivalent to the notation of meager sets in a locally compact group. In an abelian Polish group $X$ he called a set $A\subset X$ \textit{Haar meager} if there is a Borel set $B\subset X$ with $A\subset B$, a compact metric space $K$ and a continuous function $f:K\to X$ such that $f^{-1}(B+x)$ is meager in $K$ for all $x\in X$.
He also proved that the family $\mathcal{HM}$ of all Haar meager sets is a $\sigma$--ideal.

Christensen \cite[Theorem 2]{Ch} proved that $0\in\inn (A-A)$ for each universally measurable set $A$, which is not Haar null in an abelian Polish group $X$. In fact this theorem generalizes Steinhaus' theorem. Following Christensen's idea, P. Fischer and Z. S{\l}odkowski \cite{FS} introduced the notation of Christensen measurability and, consequently,  they proved that each Christensen measurable homomorphism is continuous.

The main aim of the paper is to prove analogous results for Haar meager sets.

\section{A generalization of Piccard's theorem}

First,
we show a generalization of Piccard's theorem (see e.g. \cite[Theorem~2.9.1]{Piccard}).
The following fact will be useful in the sequel:

\begin{theorem}{\em \cite[p. 90]{DS}}
 If $X$ is an abelian Polish group, there exists an equivalent complete metric $\rho$ on $X$, which is invariant; i.e. $\rho(x,y)=\rho(x+z,y+z)$ for every $x,y,z\in X$.
\end{theorem}

\begin{theorem}\label{A-A}
Let $X$ be an abelian Polish group. If $A\subset X$ is a Borel non-Haar meager set, then $0\in\inn (A-A).$
\end{theorem}

\begin{proof}
For the proof by contradiction suppose that $0\not\in\inn (A-A).$ Let
$$
F(A):=\{x\in X: (x+A)\cap A\not\in\mathcal{HM}\}.
$$
Clearly $0\in F(A)$ and $F(A)\subset A-A$. Hence $0\not\in\inn F(A)$ and there is a sequence $(x_i)_{i\in\mathbb{N}}\subset X\setminus F(A)$ such that
$$
\rho(0,x_i)\leq 2^{-i}\;\;\mbox{for each}\;\;i\in\mathbb{N}.
$$

Now, let
$$
A_0=A\setminus\left[\bigcup_{i\in\mathbb{N}}(x_i+A)\cap A\right].
$$
Since $x_i\not\in F(A)$ for each $i\in\mathbb{N}$, so $(x_i+A)\cap A\in\mathcal{HM}$ for $i\in\mathbb{N}$. Hence, by \cite[Theorem 2.9]{Darji}, $\bigcup_{i\in\mathbb{N}}(x_i+A)\cap A\in\mathcal{HM}$ and, consequently, $A_0\not\in\mathcal{HM}$, because $A\not\in\mathcal{HM}$.

Since $A_0\not\in\mathcal{HM}$, for every compact metric space $K$ and continuous function $g:K\to X$ there is a $y_K\in X$ such that $g^{-1}(A_0+y_K)$ is comeager in $K$.

Let $K:=\{0,1\}^{\aleph_0}$ be the countable Cantor cube. It is a compact metric space with the product metric
$$
d(k,l):=\sum_{i=1}^{\infty} 2^{-i}\overline{d}(k_i,l_i)\;\;\mbox{for every}\;\;k=(k_i)_{i\in\mathbb{N}},l=(l_i)_{i\in\mathbb{N}}\in K, $$
where $\overline{d}$ is the discrete metric in $\{0,1\}$. Define a function $g:K\to X$ as follows:
\begin{equation}\label{g}
g(k)=\sum_{i=1}^{\infty} k_ix_i\;\;\mbox{for}\;\; k=(k_i)_{i\in\mathbb{N}}\in K.
\end{equation}

To show that $g$ is well defined, we have to prove that the series $\sum_{i=1}^{\infty} k_ix_i$ is convergent. Indeed, for each $m,n\in\mathbb{N}$ with $m>n$ we have:
\begin{equation}\label{gx}
\begin{array}{lcl}
\displaystyle\rho\left(\sum_{i=1}^m k_ix_i,\sum_{i=1}^n k_ix_i\right)&=&\displaystyle \rho\left(0,\sum_{i=n+1}^m k_ix_i\right)\\[1ex]
&\leq& \displaystyle\rho \left(0,k_{n+1}x_{n+1})+\rho(k_{n+1}x_{n+1},\sum_{i=n+1}^m k_ix_i\right)\\[1ex]
&=&\displaystyle \rho\left(0,k_{n+1}x_{n+1}\right)+\rho\left(0,\sum_{i=n+2}^m k_ix_i\right)\\[1ex]
&\leq&\displaystyle \ldots\leq\sum_{i=n+1}^m \rho\left(0,k_ix_i\right)\leq \sum_{i=n+1}^m 2^{-i}.
\end{array}
\end{equation}
Consequently, since $K$ is complete, the series $\sum_{i=1}^{\infty} k_ix_i$ is convergent.

Next, observe that $g$ is uniformly continuous. In fact, if we fix a positive integer $n$, then, by~\eqref{gx}, we get
$$
\rho\left(\sum_{i=1}^m k_ix_i,\sum_{i=1}^n k_ix_i\right)\leq\sum_{i=n+1}^m2^{-i}<2^{-n}\;\;\mbox{ for }\,m>n
$$
for each $k=(k_n)_{n\in\mathbb{N}}\in K$. Consequently
\begin{equation}\label{gy}
\rho\left(\sum_{i=1}^{\infty} k_ix_i,\sum_{i=1}^n k_ix_i\right)\leq 2^{-n}.
\end{equation}
Fix $\varepsilon>0$. Then we can find a positive integer $N$ such that $2^{-N+1}<\varepsilon$. Further, let $0<\delta<2^{-N}$ and $k=(k_i)_{i\in\mathbb{N}},l=(l_i)_{i\in\mathbb{N}}\in K$ be such that $d(k,l)<\delta$. Since $\delta<2^{-N}$, so $k_i=l_i$ for all $i\in\{1,\ldots,N\}$. Thus
$$
\rho\left(\sum_{i=1}^N k_ix_i,\sum_{i=1}^N l_ix_i\right)=0.
$$
Hence, on account of~\eqref{gy}, we obtain
$$
\begin{array}{lll}
\lefteqn{\rho(g(k),g(l))=\rho\left(\sum_{i=1}^{\infty}k_ix_i,\sum_{i=1}^{\infty} l_ix_i \right)}\null\\[2ex]
& & \null\displaystyle \leq \rho\left(\sum_{i=1}^{\infty}k_ix_i,\sum_{i=1}^N k_ix_i \right)+ \rho\left(\sum_{i=1}^Nk_ix_i,\sum_{i=1}^N l_ix_i \right)+ \rho\left(\sum_{i=1}^Nl_ix_i,\sum_{i=1}^{\infty} l_ix_i \right)\\[3ex]
& & \null\displaystyle \leq2^{-N}+0+2^{-N}<\frac{\varepsilon}{2}+ \frac{\varepsilon}{2}=\varepsilon.
\end{array}
$$
In this way we have proved that $g$ is uniformly continuous on~$K$.

Now, for the countable Cantor cube $K$ and function $g$ given by \eqref{g} the set $g^{-1}(A_0+y_K)$ is comeager in $K$ for some $y_K\in X$. Moreover, since $A$ is a Borel set, so do $A_0$ and $g^{-1}(A_0+y_K)$. Thus the set $g^{-1}(A_0+y_K)$ is comeager with the Baire property in $K$ and we can apply the well known Piccard's theorem. Hence there is an open ball $K(0,2^{-k})$ with some $k\in\mathbb{N}$ such that
$$
K(0,2^{-k})\subset g^{-1}(A_0+y_K)-g^{-1}(A_0+y_K).
$$
Let $n=k+1$ and $e_{n}=(0,\ldots,0,1,0,0,\ldots)$, where $1$ is in the $n$-th place. Since
$d(0,e_{n})=2^{-n}<2^{-k}$, so $e_{n}\in K(0,2^{-k})$. Hence $e_{n}=a-b$ for some $a,b\in g^{-1}(A_0+y_K)$. Then
$$
g(a)-g(b)=g(e_{n}+b)-g(b)=g(e_{n})=x_{n}
$$
and whence $x_{n}\in (A_0-A_0).$ Consequently $(x_{n}+A_0)\cap A_0\neq\emptyset$, what contradicts the definition of $A_0$ and ends the proof.
\end{proof}

Let $X$ be an abelian Polish group, $\mathcal{B}(X)$ be the Borel $\sigma$--algebra on $X$ and denote by $\mathcal{F}(X)$ the family of all sets $A\subset X$ such that
$$
\forall_{\;K\subset X\mbox{\small{-compact}}}\;\; \exists _{\;x_K\in X}\;\;K+x_K\subset A.
$$
Each set $A\in \mathcal{F}(X)\cap \mathcal{B}(X)$ is neither Haar null (in view of Ulam's theorem), nor Haar meager.

Following by E.~Matou\u{s}kov\'{a} and L.~Zaj\'{\i}\u{c}ek \cite{Mat},
we can give an example of a non-Haar null universally measurable set $A\subset X$ such that $\inn (A+A)=\emptyset$; e.g. $$A=\{(x_n)_{n\in\mathbb{N}}\in c_0: \forall_{n\in\mathbb{N}}\;x_n\geq0\}\in\mathcal{F}(c_0).$$
It means that the "strong version" of Piccard's theorem as the "strong version" of Steinhaus' theorem do not hold in non-locally compact abelian Polish groups.

\section{$D$-measurability}

Let $X$ be an abelian Polish group. Let us introduce the following

\begin{definition}
{\em A set $A\subset X$ is  $D$--\textit{measurable} if $A=B\cup M$ for a Haar meager set $M\subset X$ and a Borel set
$B\subset X$.}
\end{definition}

This definition is analogous to the definition of Christensen measurability from \cite{FS}. That is why proofs of theorems in this chapter run in an analogous way as some proofs in \cite{FS}.

\begin{theorem}
The family $\mathcal{D}$ of all $D$--measurable sets is a $\sigma$--algebra.
\end{theorem}

\begin{proof}
Let $W_n\in \mathcal{D}$ for each $n\in\mathbb{N}$, i.e. $W_n=B_n\cup M_n$ for Borel sets $B_n$ and Haar meager sets $M_n$. Then, by \cite[Theorem 2.9]{Darji}, $\bigcup_{n\in\mathbb{N}}M_n$ is Haar meager and, consequently $\bigcup_{n\in\mathbb{N}}W_n=\left(\bigcup_{n\in\mathbb{N}}B_n\right)\cup \left(\bigcup_{n\in\mathbb{N}}M_n\right)\in\mathcal{D}$.

Let $W\in\mathcal{D}$, i.e.  $W=B\cup M$ for a Borel set $B$ and a Haar meager set $M$. Thus there is a Haar meager Borel set $N\supset M$. In view of \cite[Theorem 2.9]{Darji}, $M_0:=M\setminus B\subset M$ is Haar meager. Since $N_0:=N\setminus B\subset N$, $N_0$ is a Borel Haar meager set. Moreover, $$
M_0\subset N_0,\;\;M_0\cap B=\emptyset,\;\;\;N_0\cap B=\emptyset,\;\;W=B\cup M_0.
$$
Consequently,
$$
X\setminus W=X\setminus (B\cup M_0)=[X\setminus (B\cup N_0)]\cup (N_0\setminus M_0).
$$
But $N_0\setminus M_0$ is Haar meager and $X\setminus (B\cup N_0)$ is a Borel set, so $X\setminus W\in\mathcal{D}$, what ends the proof.
\end{proof}

Having $\sigma$--algebra, we can define a measurable function in the classical way:

\begin{definition}
 {\em Let $X$ be an abelian Polish group and $Y$ be a topological group. A mapping $f:X\to Y$ is a $D$--\textit{measurable function} if $f^{-1}(U)\in\mathcal{D}$ in $X$ for each open set $U\subset Y$.}
\end{definition}

Now we can prove the following
\begin{theorem}
Let $X$, $Y$ be abelian Polish groups. If $f:X\to Y$ is a $D$--measurable homomorphism, then $f$ is continuous.
\end{theorem}

\begin{proof}
Let $Y_0=\overline{f(X)}$. Clearly $Y_0$ is a subgroup of $Y$, so it is enough to prove that $f:X\to Y_0$ is continuous.

Let $U$ be a neighborhood of $0$ in $Y_0$. Then there is a neighborhood $V$ of $0$ such that $V-V\subset U$. Moreover,
$\overline{f(X)}\subset f(X)+V=Y_0$, so
$$
Y_0=\bigcup_{x\in X} (f(x)+V).
$$
Since $X$ is separable, we can choose a sequence $(x_i)_{i\in\mathbb{N}}$ such that
$$
Y_0=\bigcup_{i\in \mathbb{N}} (f(x_i)+V).
$$
Thus
$$
X=f^{-1}(Y_0)=\bigcup_{i\in \mathbb{N}} f^{-1}(f(x_i)+V)=\bigcup_{i\in \mathbb{N}} (x_i+f^{-1}(V)).
$$

The set $X$ is not Haar meager, so, according to \cite[Theorem~2.9]{Darji}, the set $x_{i_0}+f^{-1}(V)$ is not Haar meager for some $i_0\in\mathbb{N}$, so does $f^{-1}(V)\not\in\mathcal{HM}$. Hence, since $f$ is $D$--measurable, $f^{-1}(V)=B\cup M$, where $B$ is a Borel non-Haar meager set and $M$ is Haar meager. Thus, in view of Theorem~\ref{A-A}, there is a neighborhood $W\subset X$ of $0$ such that
$$
W\subset A-A\subset f^{-1}(V)-f^{-1}(V)\subset f^{-1}(V-V)\subset f^{-1}(U).
$$

In this way we proved that $f$ is continuous at $0$ and, consequently we obtain the
thesis.
\end{proof}

Here we presented some similarities between Haar meager sets and Haar null sets, as well as between $D$--measurability and Christensen measurability. We can say that the notation of Haar meager sets and $D$--measurability is a topological analog of the notation of Haar null sets and Christensen measurability in an abelian Polish group; similarly as measure and category in a locally compact topological group. But we know that there are also differences between measure and category in locally compact group (see \cite{Oxtoby}). That is why the following problem seems to be interesting:

\begin{problem}
{\em To give examples of differences between Haar meager sets and Haar null sets, as well as $D$--measurability and Christensen measurability.}
\end{problem}


\begin{thebibliography}{99}

\bibitem{Ch}
J.P.R. Christensen, {\em On sets of Haar measure zero in abelian Polish groups}, Israel J. Math. \textbf{13} (1972), 255-260.

\bibitem{Darji}
U.B. Darji, {\em On Haar meager sets}, Topology Appl. \textbf{160} (2013), 2396-2400.

\bibitem{DS}
N. Dunford, J.T. Schwartz, {\em Linear Operators I}, Interscience, 1967.

\bibitem{FS}
P. Fischer, Z. S{\l}odkowski, {\em Christensen zero sets and measurable convex functions}, Proc. Amer. Math.Soc. \textbf{79} (1980), 449-453.

\bibitem{HSY}
B.R.~Hunt, T.~Sauer, J.A.~Yorke, {\em Prevalence: a translation-invariant "almost every" on infinite-dimensional spaces}, Bull. Amer. Math. Soc. \textbf{27} (1992), 217-238.

\bibitem{HSY1}
B.R.~Hunt, T.~Sauer, J.A.~Yorke, {\em Prevalence: an addendum}, Bull. Amer. Math. Soc. \textbf{28} (1993), 306-307.


\bibitem{Mat}
E. Matou\u{s}kov\'{a}, L. Zaj\'{\i}\u{c}ek, {\em Second order differentiability and Lipschitz smooth points ofconvex functionals}, Czechoslovak Math. J. \textbf{48} (1998), 617-640.

\bibitem{Myc}
J. Mycielski, {\em Unsolved problems on the prevalence of ergodicity, instability and algebraic
independence}, Ulam Quarterly \textbf{1} (1992), 30-37.

\bibitem{Piccard}
M. Kuczma, An Introduction to the Theory of Functional
Equations and Inequalities. Cauchy's Equation and Jensen's
Inequality,  In: A.~Gil\'{a}nyi (ed.), 2nd edition, Birkh\"{a}user Verlag, Basel, 2009.


\bibitem{Oxtoby}
J.C. Oxtoby, {\em Measure and Category}, Springer-Verlag, New York-Heidelberg-Berlin, 1971.

\end{thebibliography}
\end{document}